\documentclass{article}
\usepackage[english]{babel}
\usepackage[a4paper,top=2cm,bottom=2cm,left=3cm,right=3cm,marginparwidth=1.75cm]{geometry}
\usepackage{amsmath,amssymb,amsthm}
\usepackage[colorlinks=true,allcolors=blue]{hyperref}

\usepackage{currfile}
\usepackage{navigator}
\IfFileExists{\currfilename}{\embeddedfile{sourcefile}{\currfilename}}{}

\makeatletter
\def\@seccntformat#1{\csname the#1\endcsname.\ } % the column after a section number
\makeatother

\newtheorem{theorem}{Theorem}
\newtheorem{proposition}{Proposition}
\newtheorem{lemma}{Lemma}

\theoremstyle{definition}
\newtheorem{definition}{Definition}
\theoremstyle{remark}
\newtheorem{remark}{Remark}

\newcommand{\wt}{\mathrm{wt}}
\newcommand{\Aut}{\mathrm{Aut}}
\newcommand{\FF}{\mathbb{F}}
\newcommand\GL[2]{\mathrm{GL}(#1,#2)}
\newcommand\QD[1]{\mathrm{QD}_{#1}}
\newcommand\vx[1]{#1}

\date{}

\title{On the existence of some completely regular codes\\ in Hamming graphs%
\thanks{%
The study was funded by the Russian Science Foundation, grant 22-11-00266, \url{https://rscf.ru/en/project/22-11-00266/}
}
}
\author{Denis S. Krotov\thanks{Sobolev Institute of Mathematics}}

\begin{document}
\maketitle
\begin{abstract}
We solve some first questions in the table of small parameters of completely regular (CR) codes in Hamming graphs $H(n,q)$. The most uplifting result is the existence of a $\{13,6,1;1,6,9\}$-CR code in $H(n,2)$, $n\ge 13$. We also establish the non-existence of a $\{11,4;3,6\}$-code and a $\{10,3;4,7\}$-code in $H(12,2)$ and $H(13,2)$. A partition of the complement of the quaternary Hamming code of length~$5$ into $4$-cliques is found, which can be used to construct completely regular codes with covering radius $1$ by known constructions. Additionally we discuss the parameters $\{24,21,10;1,4,12\}$ of a putative completely regular code in $H(24,2)$ and show the nonexistence of such a code in $H(8,4)$.
\end{abstract}

%(we also establish the uniqueness of such a code in $H(n,2)$, $n\ge 13$)
%and a $\{10,2;1,8\}$-CR code in $H(n,3)$, $n\ge 5$.
\tableofcontents

\section{Introduction}

In this paper, we consider the existence
of several putative non-linear completely
regular codes in Hamming graphs,
including some first open cases
in the tables~\cite{CRCtable}
of small parameters of completely regular codes.

\section{Main definitions and background}

\begin{definition}[\bf Hamming graph, $n$-cube, hypercube, code, linear code]
The \emph{Hamming graph} $H(n,q)$ is the Cartesian product of $n$ copies of~$K_q$,
the {complete graph} of order~$q$.
According to the definition of the Cartesian product, the vertex set of $H(n,q)$
consists of all $n$-words in some alphabet of order~$q$, which,
if $q$ is a prime power, is often associated with the elements of the Galois field
$\FF_q=\mathrm{GF}(q)$ of order~$q$.
In this case, the vertex set of $H(n,q)$ forms
an $n$-dimensional vector space~$\FF_q^n$ over~$\FF_q$.
% and $H(n,q)$ is a Cayley graph on its additive group~$\FF_q^{n+}$.
The binary ($q=2$) Hamming graph $H(n,2)$ is also called the \emph{$n$-cube},
or a~\emph{hypercube}.
A set of vertices of $H(n,q)$ is called
a \emph{$k$-face} if it induces a subgraph isomorphic to $H(k,q)$.

\begin{definition}[\bf weight, weight distribution]\label{d:weight}
The \emph{weight} $\wt(\vx x)$ of a vertex $\vx x$ of $H(n,q)$
is the number of nonzero symbols in
the corresponding word. The \emph{weight distribution} of a set $C$ of vertices of $H(n,q)$ is the $(n+1)$-tuple
$(A_0,\ldots,A_n)$, where $A_i$ is the number of weight-$i$ vertices in~$C$.
\end{definition}

The \emph{(Hamming) distance} $d(x,y)$ between two vertices~$x$ and~$y$
in $H(n,q)$ is the length of the shortest path  between these vertices, or,
which is the same, the number of positions
in which the words~$x$ and~$y$ differ.
A set $C$ of at least two vertices of $H(n,q)$
is called a \emph{distance-$d$ code},
if the minimum distance between
two different elements of~$C$ is~$d$.
Such a code is called a \emph{linear code},
or an \emph{$[n,\log_q|C|,d]_q$-code},
if it is a vector subspace of~$\FF_q^n$.
\end{definition}

\begin{definition}[\bf equitable partition]
An ordered partition $(C_0, \ldots, C_{k-1})$
of the vertex set of a graph~$\Gamma$ into nonempty sets
$C_0$, \ldots, $C_{k-1}$ is called \emph{equitable},
or an \emph{equitable $k$-partition},
if there is a $k$-by-$k$
matrix $(S_{i,j})_{i,j\in \{0,...,k-1\}}$
(the \emph{quotient} matrix)
such that every vertex of $C_i$
has exactly~$S_{i,j}$ neighbors
in~$C_j$, ${i,j\in \{0,...,k-1\}}$.
\end{definition}

\begin{definition}[\bf covering radius, completely regular code, CR-code]\label{d:crc}
Given a nonempty set~$C$ of vertices of a graph~$\Gamma=(V,E)$,
the \emph{distance partition} with respect to~$C$
is the partition $(C^{(0)},C^{(1)},\ldots ,C^{(\rho )})$ of~$V$
with respect to the distance from~$C$: $C^{(i)}=\{\vx x\in V:
\ d(\vx x,C)=i\}$. The number~$\rho=\rho(C)$ of
non-empty cells different from~$C$
in the distance partition
is called the \emph{covering radius} of~$C$.
A nonempty set~$C$ (code) of vertices of a graph
is called
\emph{completely regular} (completely regular set or completely regular code)
if the partition of the vertex set
with respect to the distance
from~$C$ is equitable.
The quotient matrix $(S_{i,j})_{i,j=0}^\rho$ corresponding
to a completely regular code
is always tridiagonal, and for regular graphs it is often written
in the form of the \emph{intersection array}
$$\{b_{0},b_{1},\ldots,b_{\rho-1}; c_{1},c_{2},\ldots,c_{\rho}\}=\{S_{0,1},S_{1,2},\ldots,S_{\rho-1,\rho}; S_{1,0},S_{2,1},\ldots,S_{\rho,\rho-1}\}$$
(the diagonal elements of the quotient matrix are uniquely determined from
the intersection array and the degree of the graph).
For brevity, we will use the abbreviation \emph{$S$-CR code} for ``completely regular code with intersection array~$S$'' or ``with quotient matrix~$S$''.
\end{definition}
\begin{lemma}[see {\cite[Theorem~1]{Vas09:inter}} for $q=2$, {\cite[Section~7]{Kro:struct}} in general]\label{l:loc}
Assume that in $H(n,q)$ we have an equitable
partition $(C_0, ..., C_{k-1})$,
$k$-face $F$ and $(n-k)$-face $F'$
such that $F\cap F'=\{\bar 0\}$.
The weight distributions of
$C_i\cap F$, $i=0, ...,k-1$, are uniquely
determined by
\begin{itemize}
    \item the parameters $n$, $q$ of the Hamming graph and the parameter $k$ of the faces,
    \item the quotient matrix of the equitable partition,
    \item the weight distributions of
$C_i\cap F'$, $i=0, ...,k-1$,
\end{itemize}
and do not depend on the particular choice
of the equitable partition
$(C_0, ..., C_{k-1})$
or the faces $F$, $F'$.
\end{lemma}
\begin{lemma}\label{l:sub}
Let $D$ be a set of vertices of $H(n,q)$
such that
\begin{itemize}
    \item[\rm (a)]
every vertex not in $D$ has at most $1$ neighbor in $C$.
\end{itemize}
Then the subgraph of $H(n,q)$
induced by~$D$ has connected components
isomorphic to $H(k,q)$
for some $k$ (which can be different for different components).
Moreover, each two of these components
are at distance at least~$3$ from each other.
\end{lemma}
\begin{proof}
Without loss of generality,
we assume that $\bar 0=00...0$ belongs to $D$ and has the largest degree~$k$ in the connected component~$D_0$ induced
by~$D$ and containing $\bar 0$.
If a word $\vx x$ with the only nonzero
in the $i$th position is in $D$,
then every other word $\vx y$ with the only nonzero
in the $i$th position is also in $D$
(indeed, it has at least~$2$ neighbors $\bar 0$ and $\vx x$
in $D$ and by hypothesis (a) must belong to~$D$).
We conclude that there is a set $I$ of indices such that
$\vx v$ is the $D$-neighbor of $\bar 0$ if and only if
$\vx v$ has the only nonzero symbol in some $i$th position where
$i\in I$. We will prove the following statement,
which is obviously equivalent to the first claim of the lemma.
\begin{itemize}
    \item [(*)] a vertex $\vx z$ belongs to the component
    $D_0\ni \bar 0$ if and only if  it has zeros out of $I$.
\end{itemize}
 The ``if'' part is proven by induction on the weight $w$ of $\vx z$. We already have (*) for $w=0$ and $w=1$. If $w\ge 2$,
 then $\vx z$ has exactly $w$ neighbors of weight~$w-1$. By the induction hypothesis, they belong to $D_0$; by hypothesis (a) of the lemma, $\vx x$ cannot be out of $D$ (and of course it is in the same connected component as its neighbors of weight~$w-1$). This finishes the induction step and proves the ``if'' part of (*).

 The ``only if'' part of (*) is straightforward from the ``if'' part and the assumption
 that $\bar 0$ has the maximum number~$k$ of $D$-neighbors
 among $D_0$.
 Indeed, every vertex from the set $D_0'$ of all $q^k$ vertices satisfying (*) already has~$k$ neighbors in $D_0'$; hence, $D_0'=D_0$.

We have proved the first claim of the lemma. The second claim is straightforward from (a).
\end{proof}

Note that Lemma~\ref{l:sub} is applicable to a
$\{b_{0},b_{1},\ldots,b_{\rho-1}; c_{1},c_{2},\ldots,c_{\rho}\}$-CR code $C$ if $c_1=1$ or to the last component $C^{(\rho)}$ of its distance partition if $b_{\rho-1}=1$. In the first case, the inner degree $k$ is the same for all components and equal to $n-b_{0}/(q-1)$; in the last case, to $n-c_{\rho}/(q-1)$.

\section{CR codes from perfect codes}
It is straightforward
that the maximum-weight codewords
of the ternary Golay $[11,6,5]_3$ code
and the ternary
extended Golay $[12,6,5]_3$ code
form a binary (in the alphabet $\{1,2\}$)
code of cardinality~$24$
and distance~$5$ and~$6$, respectively.
Such binary codes, the punctured Hadamard code of length~$11$ and the Hadamard code of length~$12$, are known to be completely regular and unique up to equivalence.
In this section we find that two more codes
can be obtained from linear perfect (Hamming)
codes by alphabet retraction.

\begin{remark}
   In general,
the set of maximum-weight codewords
of the extended Hamming $[n=\frac{q^m-1}{q-1},n-m,3]_q$ code in $H(n,q)$
is not necessarily a CR code in $H(n,q-1)$.
It was checked for $(m,q)$ from $(2,4)$, $(3,4)$, $(4,4)$, $(5,4)$, $(3,3)$, $(4,3)$, $(5,3)$, $(2,5)$, $(2,7)$, $(2,8)$, $(2,9)$, and only $(2,4)$
and $(3,3)$ induced a CR code in such a way.
\end{remark}

\subsection{\{13,6,1;1,6,9\}-CR, H(13+,2)}
\subsubsection{Observations}
Let $C$ be a $\{13,6,1;1,6,9\}$-CR code.
$(C_0$, $C_1$, $C_2$, $C_3)$ --- equitable distance partition with respect to $C=C_0$. The quotient matrix is
$Q=\begin{pmatrix}
    0&13&0&0\\1&6&6&0\\0&6&6&1\\0&0&9&4
\end{pmatrix}.$
Because of $1$ in the intersection array,
by Lemma~\ref{l:sub},
the vertices of $C_3$
form $26$ $4$-faces at distance at least~$3$ from each other, they are divided into $13$ antipodal pairs of $4$-faces.
Two such $4$-faces have at least~$1$ common direction (otherwise, the corresponding two antipodal pairs of $2$-faces are at distance at least~$2$ from each other, a contradiction).
%For such a $4$-face, the distribution of colors in the orthogonal $9$-face is
%\begin{verbatim}
%[ 0  0  0  1 ]
%[ 0  0  9  0 ]
%[ 0 27  9  0 ]
%[ 9 36 36  3 ]
%[ 0 54 63  9 ]
%[ 0 54 63  9 ]
%[ 9 36 36  3 ]
%[ 0 27  9  0 ]
%[ 0  0  9  0 ]
%[ 0  0  0  1 ]
%\end{verbatim}
The weight distribution of colors with respect to a vertex from $C$ is
\begin{verbatim}
[  1   0   0   0 ]
[  0  13   0   0 ]
[  0  39  39   0 ]
[ 13 104 156  13 ]
[ 26 325 312  52 ]
[ 39 624 546  78 ]
[ 65 767 819  65 ]
[ 65 767 819  65 ]
[ 39 624 546  78 ]
[ 26 325 312  52 ]
[ 13 104 156  13 ]
[  0  39  39   0 ]
[  0  13   0   0 ]
[  1   0   0   0 ]
\end{verbatim}

\subsubsection{Structure}

\begin{proposition}
Let  the code $D$ be formed
by all cyclic permutations
of words of form
$000100ab01c1d$,
$111011ab10c0d$,
where $a,b,c,d\in\{0,1\}$.
The code $D$ is completely regular with
intersection array
$\{9,6,1;1,6,13\}$.

The elements of the distance partition
$(D=D^{(0)},D^{(1)},D^{(2)},D^{(3)}=C)$
have cardinalities $13\cdot 32$,
$9\cdot 13\cdot 32$,
$9\cdot 13\cdot 32$, $9\cdot 32$,
respectively. The automorphism group
of the partition (as well as of each of its components)
has order $11232$, is isomorphic to $\GL33$,
and acts transitively on each $D^{(i)}$,
$i=0,1,2,3$.
The permutation automorphism group
has order $39$ and generated
by the coordinate permutations
%$(0,1,2,3,4,5,6,7,8,9,10,11,12)$
$(1,2,3,4,5,6,7,8,9,10,11,12,13)$
and
 $(1,2,5)(3,8,10)(4,11,6)(9,13,12)$.
%$(0,1,4)(2,7, 9)(3,10,5)(8,12,11)$.
\end{proposition}

There are $78$ order-$288$ subgroups
of $\Aut(H(13,2))$ that act regularly
on~$C$ (existing of such a group means that $C$ is \emph{propelinear}
in the sense of~\cite{RifPuj:1997}).
These subgroups form two conjugate classes, with the representatives given in Table~\ref{tab:0},
\begin{table}[ht]
    \centering
    $
    \begin{array}{l}
g_0=(
1111111111111,
%[0,1,2,3,4,5,6,7,8,9,10,11,12],
\mathrm{Id}
)
\\
g_1=(0001010011000,
%[1,4,3,8,0,9,6,7,2,11,10,5,12],
(1,2,5)(3,4,9)(6,10,12)) \\
g_2=(0011000000101,
%[1,4,8,2,0,5,6,10,3,9,12,11,7],
(1,2,5)(3,9,4)(8,11,13)) \\
g_3=(0001100100000,
%[3,8,9,5,2,12,6,1,11,7,4,10,0],
(1,4,6,13)(2,9,12,11,5,3,10,8)) \\
g_4=(1011010010000,
%[5,11,8,3,9,0,6,7,2,4,10,1,12],
(1,6)(2,12)(3,9)(5,10)) \\
\end{array}
\quad
    \begin{array}{l}
g_0=(
1111111111111,
%[0,1,2,3,4,5,6,7,8,9,10,11,12],
\mathrm{Id}
)
\\
g_1=(
1010000010001,
%[2,8,10,7,3,5,6,4,12,9,0,11,1],
(1,3,11)(2,9,13)(4,8,5)
)
\\
g_2=(
0011000000101,
%[1,4,8,2,0,5,6,10,3,9,12,11,7],
(1,2,5)(3,9,4)(8,11,13)
)
\\
g_3=(
0000000011001,
%[4,3,2,0,7,6,9,8,10,11,12,5,1],
(1,5,8,9,11,13,2,4)(6,7,10,12)
)
\\
g_4=(
0100000000110,
%[7,10,2,3,12,9,6,0,8,5,1,11,4],
(1,8)(2,11)(5,13)(6,10)
)
\end{array}$
    \caption{Generators of two propelinear structures on $C$}
    \label{tab:0}
\end{table}
where each representative $G$ is given by
generators
$g_0$, $g_1$, $g_2$, $g_3$, $g_4$
of order $2$, $3$, $3$, $8$, $2$,
respectively where
$$G =
\langle g_0 \rangle
\times
\big(
\langle g_1,g_2 \rangle
\rtimes
\langle g_3,g_4 \rangle
\big)
\sim C_2 \times\big((C_3\times C_3)\rtimes \QD{16}\big)
.$$

\begin{theorem}
The maximum-weight codewords
of the Hamming $[13,10,3]_3$ code in $H(13,3)$
form a $\{13,6,1;1,6,9\}$-CR code
in the $H(13,2)$-subgraph of $H(13,3)$ induced by the weight-$13$ vertices.
\end{theorem}
\begin{proof}
 The proof is computational; direct check.
 % \alert{TBA}
\end{proof}

\subsubsection{Uniqueness}

\begin{theorem}\label{th:13unique}
 For each $n=13$,
 % \alert{complete  $n\ge 13$,}
in $H(n,2)$
there is a unique
(up to equivalence)
$\{13,6,1;1,6,9\}$-CR code.
\end{theorem}
\begin{proof}
For $n=13$, the proof involves computation.
We start with some observations.
Let $(C,C^{(1)},C^{(2)},C^{(3)})$
be the distance partition of the vertex set of $H(13,2)$ with respect to a
$\{13,6,1;1,6,9\}$-CR code $C$.
W.l.o.g., we assume $\bar 0\in C$.
By Lemma~\ref{l:sub}, $C^{(3)}$ consists
of $26$ connected components, $4$-faces.
Each such face is at distance~$3$ from~$C$
and at least~$3$ from any other such face.
By weight-distribution formulas
(see e.g.~\cite{Kro:struct}), we find
that there are exactly~$13$
weight-$3$ codewords of~$C$ and~$13$
weight-$3$ codewords of~$C^{(3)}$.
Since all these $26$ words are at distance at least~$3$ from each other (actually, at least~$4$ because they have the same weight),
they form a constant-weight binary code with following parameters: length~$13$, weight~$3$, minimum Hamming distance~$4$, size~$26$. It is known that there are exactly~$2$ nonequivalent such codes (the supports of the codewords form a configuration known as a Steiner triple system of order~$13$, STS$(13)$).
For each code $D$ from these two codes, we apply the following algorithm.

Firstly, we construct the set of candidates for faces containing a weight-$3$ word. Each such a candidate contains exactly~$1$ codeword of~$D$ and is at distance at least~$3$ from any other codeword of~$D$.

Secondly, in the set $S$ of candidates
we need to find $13$ faces at distance at least~$3$ from each other.
To do this we construct the \emph{compatibility graph} on~$S$,
where two faces are adjacent if and only if the distance between them is at least~$3$,
and search for a clique of size~$13$
in this graph (see \cite{cliquer}).
It happens that there is a unique solution
for one of the codes~$D$ and no solutions for the other.
The remaining $13$ faces are uniquely
reconstructed using the self-complementary property of equitable partitions.
%\alert{more details}.
Having $C^{(3)}$, we uniquely find the distance partition $(C^{(3)},C^{(2)},C^{(1)},C)$ with respect to it.

% For $n>13$, we will prove that every $\{13,6,1;1,6,9\}$-CR code in $H(n,2)$ has $n-13$ non-essential coordinates; it follows that the number of equivalence classes of such codes coincides with that number in $H(13,2)$. \alert{TBA, use Lemma~\ref{l:loc}}
\end{proof}

\subsection{\{10,2;1,8\}-CR codes in H(5+,3)}

\begin{theorem}
The maximum-weight codewords
of the Hamming $[5,3,3]_4$ code in $H(5,4)$
form a $\{10,2;1,8\}$-CR code
in the $H(5,3)$-subgraph of $H(5,4)$ induced by the weight-$5$ vertices.
\end{theorem}
\begin{proof}
  The proof is computational; direct check.
  %  \alert{TBA}
\end{proof}

The $\{10,2;1,8\}$-CR code in $H(5,4)$ is not
new; it was previously found in
the database \url{http://www.pietereendebak.nl/oapackage/series.html} \cite{SEN:2010:OA} of orthogonal arrays as one (and only one) of $10$ nonequivalent OA$(18,5,3,2)$
(in particular, this means that the $\{10,2;1,8\}$-CR code in $H(5,3)$ is unique).

\begin{remark}
   The set of maximum-weight codewords
of the extended Hamming $[6,3,4]_4$ code in $H(6,4)$
is not a CR code in $H(6,3)$; it is the first cell of an equitable partition with quotient matrix
$
\begin{pmatrix}
0 &12 & 0 & 0 & 0 \\
1 & 1 & 5 & 5 & 0 \\
0 & 6 & 0 & 6 & 0 \\
0 & 4 & 4 & 2 & 2 \\
0 & 0 & 0 &12 & 0
\end{pmatrix}.
$
\end{remark}

\section{Classification with local-reconstruction algorithm}\label{s:loc}
In this section, we describe the results of the classification of some CR codes
by the approach, similar to the one in~\cite[Sect.~4]{Kro:OA13}.
In the current work, the algorithm was
adopted for the classification
of CR codes with covering radius~$2$
(comparing with covering radius~$1$ in~\cite{Kro:OA13}).
With this algorithm, described below,
we establish the nonexistence of
$\{11,4;3,6\}$-CR codes in $H(12,2)$ and $H(13,2)$
and $\{11,4;3,6\}$-CR codes in $H(12,2)$
 (Section~\ref{s:11,4}).
 Additionally, to test the algorithm of some existing objects,
 we classify $\{10,5;3,6\}$-CR codes in $H(13,2)$.

The classification approach is based on the concept of local pairs.
Let $S$ be the quotient matrix of the considered putative CR codes.
First we fix the number $d\in\{0,1,2\}$, which corresponds to the distance from the all-zero word $\bar 0$ to the putative completely regular code we are going to construct. We say a pair of disjoint
sets of sets $P_0,P_2 \subset \{0, 1\}^{n}$
is an $(r_0,r_2)$-local pair if
\begin{itemize}
\item[(I)] (locality condition)
$P_i$ consists of words of weight $\le r_i$, $i=0,2$;

\item[(II)] (up-to-equivalence condition)
for $i \in\{0,2\}$,
$\bar 0\in P_i $ if and only if $d = i$;

\item[(III')] (exact cover condition) the neighborhood of every vertex $\vx v$ of weight
less than $r_0$ satisfies the local condition from the definition of an $S$-CR: if
$\vx v \in P_0$, then $\vx v$ has $S_{0,0}$ neighbors in $P_0$; if $\vx v \in P_2$, then $\vx v$ has
$S_{2,0}=0$ neighbors in $P_0$; if
$\vx v \not\in P_0\cup P_2$, then $\vx v$ has $S_{1,0}$ neighbors in $P_0$;
\item[(III'')] (exact cover condition) the neighborhood of every vertex $\vx v$ of weight
less than $r_2$ satisfies the local condition from the definition of an $S$-CR: if
$\vx v \in P_2$, then $\vx v$ has $S_{2,2}$ neighbors in $P_2$; if $\vx v \in P_0$, then $\vx v$ has
$S_{0,2}=0$ neighbors in $P_2$; if
$\vx v \not\in P_0\cup P_2$, then $\vx v$ has $S_{1,2}$ neighbors in $P_2$.
\end{itemize}

The general algorithmic approach is the following:
we choose $d$ and find the unique, up to equivalence, $(1,1)$-local pair.
On the next step, by solving the corresponding exact-covering problem with multiplicities by \texttt{libexact} \cite{KasPot08} (the details are similar to those in \cite{Kro:OA13}), we found all $(1,2)$-local (or $(2,1)$-local) pairs
that continue that unique $(1,1)$-local pair
and classify them up to equivalence (using the graph-isomorphism software \cite{nauty2014};
this step is also called \emph{isomorph rejection}).
Next, for each of the found
nonequivalent representatives,
we find all its $(2,2)$-local continuations,
then, after isomorph rejection,
 $(2,3)$- or  $(3,2)$-local continuations,
 and so on.
 Finally, as $(n,n)$-local pairs, we find all nonequivalent equitable
 partitions $(P_0,\{0,1\}^n\setminus(P_0\cup P_2),P_2)$ with quotient matrix~$S$.

 The approach, apart of the choice
 of $d$ suggests that we can choose
 between parameters $(1,2)$ and $(2,1)$,
 $(2,3)$ and $(3,2)$, $(3,4)$ and $(4,3)$, and so on, of local partitions at the corresponding step.
 Sometimes, one choice result in much faster calculations than the other,
 or even only one can be finished in reasonable time.
 To compare the timing of computation,
 different alternative choice were done
 during the study.
 This resulted in redundant but more informative computation, the results of which are shown in the tables below.

 % \alert{TBA: describe the algorithm}

\subsection{\{11,4;3,6\}-CR and \{10,3;4,7\}-CR codes}
\label{s:11,4}
In this section, we establish the nonexistence,
in the binary Hamming graphs,
of CR codes with quotient matrices
$$
U_{12}=\begin{pmatrix}1&11&0\\4&5&3\\0&6&6\end{pmatrix},
\quad
U_{13}=\begin{pmatrix}2&11&0\\4&6&3\\0&6&7\end{pmatrix},
\quad
T_{12}=\begin{pmatrix}2&10&0\\4&5&3\\0&7&5\end{pmatrix},
\quad
T_{13}=\begin{pmatrix}3&10&0\\4&6&3\\0&7&6\end{pmatrix}.
$$
\begin{theorem}[computational]
There are no $\{11,4;3,6\}$-CR codes ($\{10,3;4,7\}$-CR codes)
in $H(12,2)$ and $H(13,2)$, with quotient matrix $U_{12}$ and $U_{13}$ ($T_{12}$ and $T_{13}$), respectively.
\end{theorem}

\begin{tabular}{r|r@{\ }c|r@{\ }c|r@{\ }c|r@{\ }c|r@{\ }c}
&\multicolumn{8}{c|}{$\{11,4;3,6\}$-CR}
\\ \hline
& \multicolumn{2}{c|}{$n=12$, $\bar 0 \in C^{(0)}$}
& \multicolumn{2}{c|}{$n=12$, $\bar 0 \in C^{(2)}$}
& \multicolumn{2}{c|}{$n=13$, $\bar 0 \in C^{(0)}$}
& \multicolumn{2}{c|}{$n=13$, $\bar 0 \in C^{(1)}$}
\\ \hline
$(2,1)$-local
& $6$ & \em 13''
& $2$ & \em 0''
& $26$ & \em 6'
& $68$ & \em 1''
\\
$(1,2)$-local
& $266$ & \em 50h
& $+$ &
& $+$ &
& $77$ & \em 6''
\\
$(2,2)$-local
& $1523071$ & \em 30', 7'
& $970415$ & \em 80h
& $32546865$ & \em 4h
& $86157$ & \em 42'', 23''
\\
$(3,2)$-local
& $4$ & \em 30'
& $0$ & \em 1.2h
& $668029757$ & \em 180d
& $1855152$ & \em 22'
\\
LP
&   &
&   &
& $100261^{**}$ & \em 7d
& $14063^{**}$ & \em 2.5h
\\
$(2,3)$-local
& $14719$ & \em 1.2h
& $+$ &
&$+$&
&$+$&
\\
$(3,3)$-local
& $1$ & \em 0'', 25''
& $-$ &
& $177197$ & \em 16h
&$ 70291123 $ &  \em 145d
\\
LP
& $0^{*}$, $0^{**}$  &  \em 0'', 0''
& &
& $0^{*}$, $0^{**}$  & \em 3', 5'
\\ \hline
\end{tabular}

\begin{tabular}{r|r@{\ }c|r@{\ }c|r@{\ }c|r@{\ }c|r@{\ }c}
&\multicolumn{4}{c|}{$\{10,3;4,7\}$-CR}
\\ \hline
& \multicolumn{2}{c|}{$n=12$, $\bar 0 \in C^{(0)}$}
& \multicolumn{2}{c|}{$n=13$, $\bar 0 \in C^{(1)}$}
\\ \hline
$(2,1)$-local
& $161$ & \em 1h
& $1918$ & \em 3'
\\
$(1,2)$-local
& $21$ & \em 3'
& $17$ & \em 0''
\\
$(2,2)$-local
& $1646997$ & \em 2', 6'
& $511808 $ & \em 1', 25'
\\
LP
&  &
& $248505^*$, $215^{**}$ & \em 1.5h, 9'
\\
$(3,2)$-local
& $1299034$  & \em 1h
&$+$&
\\
$(2,3)$-local
&$287$& \em 1h
& $12771$ & \em 25'
\\
$(3,3)$-local
& $27$ & \em 35', 0'
& $ 437148986$ & \em 25h
\\
LP
& $0^{*}$, $0^{**}$  &  \em 0'', 0''
& $0$  &  \em 5d
\\ \hline
\end{tabular}

\subsection{\{10,5;3,6\}-CR codes in H(13,2)}
\label{s:10,5}
To establish the nonexistence
of some covering-radius CR codes in Section~\ref{s:11,4},
we use the algorithm of
weight-by-weight reconstruction
of the perfect coloring corresponding
to a completely regular code.
While the approach is not new, such an algorithm was never applied before
for the classification of CR codes with covering radius more than~$1$.
There are approaches to verify
the results of the classification
(for example by double-counting based on the orbit--stabilizer theorem,
see~\cite[Sect. 10.2]{KO:alg})
if the classified class is not empty.
For this reason,
we firstly classify CR codes with intersection array $\{10,5;3,6\}$,
which are known to exist in $H(13,2)$
(and also in $H(12,2)$, but $13$ is the
highest value of~$n$ for which the classification by the considered approach
is possible with reasonable computational
resources).

\begin{tabular}{r|r c | r c | r c | r c }
&\multicolumn{4}{c|}{$\{10,5;3,6\}$-CR}
\\ \hline
& \multicolumn{2}{c|}{$n=12$, $\bar 0 \in C^{(0)}$}
& \multicolumn{2}{c|}{$n=13$, $\bar 0 \in C^{(0)}$}
\\ \hline
$(2,1)$-local
& $20$ & \em 30sec
& $125$ & \em 35min
\\
$(1,2)$-local
& $+$ &
& $+$ &
\\
$(2,2)$-local
& $49192$ & \em 10sec
& $3617224$ & \em 3min
\\
$(3,2)$-local
& $261$ & \em 1min
& $33898933$ & \em 18h
\\
$(2,3)$-local
& $35693$ & \em 1min
& $+$ &
\\
$(3,3)$-local
& $240$ & \em 0, 12sec
& $1774384$ & 7days
\\
$(4,3)$-local
& $51$ & \em 0sec
& $12290300$ &  \em 16h
\\
$(3,4)$-local
& $44$ &  \em 15sec
& $+$&
\\
$(4,4)$-local
& $44$ & \em 0sec
& $4988624$ &  \em 4.5days
\\
$(5,4)$-local
& $1$ & \em 0sec
& $1449$ &  \em 1.5days
\\
$(5,5)$-local
& $1$ & \em
& $1449$ & \em 32sec
\\
$(6,5)$-local
& $1$ & \em
& $31$ & \em 21sec
\\
$(\ge 6,\ge 6)$-local
& $1$ & \em
& $31$ & \em 1sec
\\ \hline
\end{tabular}

Testing final solutions for isomorphism,
we find $9$ nonequivalent equitable partitions.
Since the number of
$\begin{pmatrix}2&5\\4&3\end{pmatrix}$-CR codes in $H(7,2)$
is also $9$, by Cartesian product with the unique
$\begin{pmatrix}1&5\\3&3\end{pmatrix}$-CR codes in $H(7,2)$
we obtain all $9$ $\{10,5;3,6\}$-CR codes in $H(13,2)$.

\section{Partition of the complement of the 4-ary Hamming code}

From \cite{BKMTV} (end of Sect. 4.6), we know
\begin{proposition}
There are $54$ partitions of the complement
of the Hamming $[5,3,3]_4$ code in $H(5,4)$
into $4$-cliques, forming one equivalence class, with the automorphism group order~$8$
of each partition.
\end{proposition}

By direct computation with the \texttt{libexact} library~\cite{KasPot08},
we find the following:
\begin{theorem}\label{th:cli}
There are $168960$ partitions of the complement
of the Hamming $[5,3,3]_4$ code in $H(5,4)$
into $4$-cliques, forming $11$ equivalence classes, with the automorphism group order
$1$, $1$, $1$, $1$, $1$, $1$,
$2$, $3$, $6$, $6$, $6$
of partitions in each equivalence class, respectively.
\end{theorem}

In \cite{BKMTV}, such partitions are used to construct CR codes with covering radius~$1$.
% In particular, from Theorem~\ref{th:cli}, we get the existence of \alert{more details TBA}

\section{\{24,21,10;1,4,12\}-CR codes
in H(8,4) and H(24,2)
}

Intersection array $[24,21,10;1,4,12]$ corresponds
to the quotient matrix
$$
\left(
\begin{array}{cccc}
 0 & 24 & 0 & 0 \\
 1 & 2  & 21 & 0 \\
 0 & 4  & 10 & 10 \\
 0 & 0 & 12 & 12
\end{array}
\right).
$$
A putative binary completely regular code $C$ with this intersection array
has minimum distance $3$ and size~$2^{16}$,
but it cannot be linear (nor $Z_4$- or $Z_2Z_4$-linear) because
there is no a distance-regular graph with this intersection array.
The $8$ weight-$3$ words of $C$ form a $S(24,3,1)$ Steiner system,
i.e., a partition of the coordinated into triples.
The $42$ weight-$4$ words of $C$ form a group divisible design, GDD,
with the $8$ groups of size $3$ correspoding to the weight-$3$
words. It is known that such GDD exist.
The weight distribution of $C$ is
\begin{align*}
  [1, 0, 0, 8, 42, 168, 448, 1344, 3279, 5120, 6720, 9744, 11788, \\9744, 6720, 5120, 3279, 1344, 448, 168, 42, 8, 0, 0, 1].
\end{align*}
The dual weight distribution is
$$ [1, 0, 0, 0, 0, 0, 0, 0, 42, 0, 0, 0, 168, 0, 0, 0, 45, 0, 0, 0, 0, 0, 0, 0, 0]. $$

\subsection{\{24,21,10;1,4,12\}-CR codes in H(8,4) do not exist}

One of the ways to construct a CR code
in $H(3n,2)$ is to construct a CR code with the same intersection array in
$H(n,4)$ and then use graph covering $H(3n,2) \to H(n,4)$
(or, alternatively, concatenation).

A putative $\{24,21,10;1,4,12\}$-CR code $C$
has code~distance $4$, weight distribution = dual weight distribution
$$ [1, 0, 0, 0, 42, 0, 168, 0, 45] $$
(there is a linear self-dual code with this weight distribution;
namely, the binary $1$st order Reed--Muller $(8,16,4)$ code lifted to GF$(4)$,
but it is not completely regular).
The codewords of weight~$4$ form a kind of $2$-design~$C_4$ in the sense
that every weight-$2$ word from $H(8,4)$ is at distance~$2$ from exactly one
element of~$C_4$.
Such designs exist (I have found at least~$2$ nonequivalent, invariant under multiplication by a constant from GF$(4)$), but it is not difficult to show that
such design necessarily contains two codewords at odd distance from each other, which contradicts the weight distribution
of~$C$. So, we have the following.

\begin{theorem}
There is no $\{24,21,10;1,4,12\}$-CR code $C$ in $H(8,4)$.
\end{theorem}
\begin{proof}
Assume w.l.o.g. that $00000000$ and $11110000$ are in $C$.
We know that

(i) the distance between two codewords is even;

(ii) every word $x$ of weight 2 is at distance 2 from exactly 1 weight-4 codeword $c$; we say that $x$ is \emph{{covered}} by $c$.

Consider the words $12000000$, $13000000$, $10200000$, $10300000$, $10020000$, $10030000$.
The support of the codeword that covers one of this words must intersect the support of the codeword
$11110000$ in exactly $3$ points, by (1). It is not difficult to find from (i) and (ii) that w.l.o.g. we have\\
$11110000$\\
$12201000$\\
$13022000$\\
$10333000$.

Now consider the word $10000100$.
The codeword that covers it, by (ii), cannot have nonzeros in position $2$, $3$, $4$, or $5$.
So, it is $100001ab$ for some $a$, $b$.

Similarly, consider the word $10000200$. It is covered by $100002a'b'$ for some $a'$, $b'$.

Now, we have two codewords $100001ab$ and $100002a'b'$ at distance less than~$4$ from each other. A contradiction.
\end{proof}

\subsection{\{24,21,10;1,4,12\} CR codes do not exist in the coset graph of the Golay
[24,12,8] code}

One of the ideas was to construct a binary $\{24,21,10;1,4,12\}$-CR as the union of cosets of the Golay
$[24,12,8]_2$ code $G$. Equivalently, to construct
a $\{24,21,10;1,4,12\}$-CR code
in the coset graph of~$G$.

The nonexistence of such codes can be shown as follows.
As was mentioned above, the weight-$3$ codewords of a binary $\{24,21,10;1,4,12\}$-CR code form a partition of the $24$ coordinates into $8$ groups of size $3$ (i.e., a $S(24,3,1)$ Steiner system).
Of course such system is unique up to equivalence.
However, there are $38233$ equivalence classes of such systems with respect to the Golay code. For each such system $C_3$, it is checked that a $\{24,21,10;1,4,12\}$-CR code in the coset graph of $G$ cannot include this $C_3$.
Namely, it is checked (computationally) that there is a vertex $x$
at distance $2$ from $C_3$ such that $|\{y\in C_3:\ d(x,y)=2|\}| > 2$
(for a $\{24,21,10;1,4,12\}$-CR code~$C$ and $x$ at distance $1$ or $2$
from $C$, it holds $|\{y\in C_3:\ d(x,y)=2|\}| \le 2$).

\section*{Acknowledgments}
The author thanks
the Supercomputing
Center of the Novosibirsk State University
for provided computational resources.

% \bibliographystyle{alpha}
% \bibliography{k}

% \end{document}

\newcommand{\etalchar}[1]{$^{#1}$}
\providecommand\href[2]{#2} \providecommand\url[1]{\href{#1}{#1}}
  \def\DOI#1{{\href{https://doi.org/#1}{https://doi.org/#1}}}\def\DOIURL#1#2{{\href{https://doi.org/#2}{https://doi.org/#1}}}

\end{document}